\documentclass[a4paper,11pt]{amsart}
\usepackage[utf8]{inputenc}
\usepackage[T1]{fontenc}
\usepackage[english]{babel}

\usepackage{microtype}
\usepackage{amsmath,amssymb,amsthm,amscd,amsfonts,mathtools}
\usepackage{stmaryrd}
\usepackage{fancyhdr}
\usepackage{hyperref}
\usepackage{mathtools}
\usepackage{a4wide}
\usepackage{tikz-cd}
\usepackage{todonotes}

\makeatletter
\@namedef{subjclassname@2020}{\textup{2020} Mathematics Subject Classification}
\makeatother

\theoremstyle{definition}
\newtheorem{exem}{Example}[section]
\newtheorem{rema}[exem]{Remark}
\newtheorem{defi}[exem]{Definition}

\theoremstyle{plain}
\newtheorem{theo}[exem]{Theorem}
\newtheorem{prop}[exem]{Proposition}
\newtheorem{coro}[exem]{Corollary}
\newtheorem{lemm}[exem]{Lemma}
\newtheorem{conj}[exem]{Conjecture}
\newtheorem{claim}[exem]{Claim}

\newcommand{\Z}{\mathbb{Z}}
\newcommand{\Q}{\mathbb{Q}}

\newcommand{\R}{\mathbb{R}}
\newcommand{\C}{\mathbb{C}}
\newcommand{\Oinf}{\mathcal{O}_{\infty}}
\newcommand{\Kinf}{K_{\infty}}
\newcommand{\Cinf}{\C_{\infty}}
\DeclareMathOperator{\rk}{\mathrm{rk}}
\DeclareMathOperator{\tr}{\mathrm{tr}}

\DeclareMathOperator{\Hom}{\mathrm{Hom}}
\newcommand{\GL}{\mathrm{GL}}
\newcommand{\PGL}{\mathrm{PGL}}

\newcommand{\Fp}{\mathbb{F}_p}
\newcommand{\Fq}{\mathbb{F}_q}
\DeclareMathOperator{\End}{\mathrm{End}}

\newcommand{\pp}{\mathfrak{p}}
\newcommand{\qq}{\mathfrak{q}}
\newcommand{\nn}{\mathfrak{n}}
\newcommand{\mm}{\mathfrak{m}}
\newcommand{\picarlitz}{\widetilde{\pi}}
\newcommand{\et}{^{\times}}
\newcommand{\bs}{\backslash}
\newcommand{\TT}{\mathbb{T}} 
\newcommand{\BT}{\mathcal{T}} 
\newcommand{\Iwa}{\mathcal{I}_{\infty}} 
\newcommand{\res}{\mathrm{res}} 
\newcommand{\PP}{\mathbb{P}} 
\newcommand{\Div}{\mathrm{Div}} 
\newcommand{\SM}{\mathbb{M}} 
\newcommand{\cusps}{\mathrm{cusps}} 
\newcommand{\tors}{_{\mathrm{tors}}} 
\newcommand{\DMF}{\mathrm{DMF}}
\newcommand{\HC}{\mathrm{HC}}

\newcommand{\pairing}{\varphi}
\newcommand{\Harm}{\mathcal{H}} 
\newcommand{\somme}{\theta} 

\newcommand{\grandmid}{\,\middle\vert\, }
\newcommand{\smallpmatrix}[4]{\left( \begin{smallmatrix}  #1&#2\\ #3&#4 \end{smallmatrix} \right)}  

\begin{document}
	
\title{Non-perfect pairings between Hecke algebra and modular forms over function fields}
\author{Cécile Armana}

\address{Université de Franche-Comté, CNRS, LmB (UMR 6623), F-25000 Besançon, France}
\email{cecile.armana@univ-fcomte.fr}

\thanks{}

\begin{abstract}
We study two analogs, for modular forms over $\Fq(T)$, of the pairing between Hecke algebra and cusp forms given by the first coefficient in the expansion. For Drinfeld modular forms, the $\Cinf$-pairing is provided by the first coefficient of their $t$-expansion at infinity. For $\Z$-valued harmonic cochains, the $\Z$-pairing is given by their Fourier coefficient with respect to the trivial ideal. We prove that, contrarily to classical cusp forms, both pairings in weight~$2$ are not perfect in a quite general setting, namely for the congruence subgroup $\Gamma_0(\nn)$ with any prime ideal $\nn$ in $\Fq[T]$ of degree~$\geq 5$. We show it by exhibiting a common element of the Hecke algebra in the kernels of both pairings and proving that it is non-zero using computations with modular symbols over $\Fq(T)$. Finally we present computational data on other kernel elements of these pairings.
\end{abstract}

\maketitle

{\small
		
}

\smallskip
\thispagestyle{empty}

\section{Introduction}

Classical cusp forms and elements of the Hecke algebra are related by the $q$-pairing $(T,f) \mapsto a_1(Tf)$ induced by the first coefficient of the $q$-expansion. This pairing is perfect over~$\C$, and over~$\Z$ for cusp forms with integral coefficients (see \cite[Thm. 2.2]{Ribet} and \cite[Lem.~1.2]{Wiese-ComputationalAMF} for instance). It is a consequence of the well-known formula on the Hecke operator~$T_n$:
\begin{equation}\label{eq-a1TnF}
\forall n\geq 1,\quad	a_n(f)=a_1(T_n f) .
\end{equation}
Through this perfectness, the space of cusp forms and its corresponding Hecke algebra are linear dual to each other, a fact that has several consequences for the study of modular forms. For instance it serves as a core ingredient in Mazur's formal immersion method for computing rational points on modular curves \cite{MazRationalisogenies}. It also provides explicit bounds for generators of the Hecke algebra, which are derived from Sturm bounds for cusp forms (see the appendix of \cite{LarioSchoof}). A full exposition, based on this duality, of computational aspects of classical modular forms over any ring is given in \cite{Wiese-ComputationalAMF}.

\smallskip
The aim of this paper is to study similar pairings for modular forms over function fields. Let~$q$ be a power of a prime number $p$, $A$ the ring $\Fq[T]$, $K$ the rational function field $\Fq(T)$. Let $\Kinf$ be the completion of $K$ with respect to the place $\infty = 1/T$ and $\Cinf$ the completion of an algebraic closure of $\Kinf$. For an ideal $\nn$ of $A$, we consider the Hecke congruence subgroup $\Gamma_0(\nn) = \left\{ \left(\begin{smallmatrix}a&b\\c&d\end{smallmatrix}\right) \in \GL_2(A) \grandmid  c\equiv 0 \bmod \nn \right\}$ of $\GL_2(A)$. Recall that there are two notions of modular forms attached to $\Gamma_0(\nn)$ in this setting:
\begin{itemize}
	\item Drinfeld modular forms, defined on $\Omega := \Cinf - \Kinf$ with values in $\Cinf$. They have a so-called $t$-expansion at infinity, indexed by non-negative integers (see Section~\ref{sec-DMF}). We denote by $M_{2,1}^{0,0}(\nn)$ the $\Cinf$-vector space of doubly cuspidal Drinfeld modular forms of weight $2$ and type~$1$ for $\Gamma_0(\nn)$.
	\item Harmonic cochains, defined on the edges of the Bruhat-Tits tree attached to $\PGL_2(\Kinf)$. Let $\Harm_{0}(\nn,\Z)$ be  the $\Z$-module  of cuspidal harmonic cochains invariant under $\Gamma_0(\nn)$ with values in $\Z$. They have a Fourier expansion indexed by non-zero ideals of $A$ (see Section~\ref{sec-HC}).
\end{itemize}
The Hecke operators $T_\pp$ are indexed by non-zero ideals $\pp$ of $A$ and they act linearly on Drinfeld modular forms and harmonic cochains. They generate the commutative Hecke algebras $\TT(\nn,\Cinf)$ and $\TT(\nn,\Z)$ which are subspaces of $\End_{\Cinf}(M_{2,1}^{0,0}(\nn))$ and $\End_{\Z}(\Harm_{0}(\nn,\Z))$ respectively.
By Gekeler--Reversat \cite{GR}, there is a Hecke-compatible isomorphism between  $\Harm_{0}(\nn,\Z)\otimes_{\Z} \Fp$ and a subspace of weight-$2$ Drinfeld modular forms with residues in $\Fp$. This provides a canonical isomorphism  $\TT(\nn,\Cinf) \simeq \TT(\nn,\Z)\otimes_{\Z} \Cinf$ (see Section~\ref{ssection-comparisonisom}). 

Taking the first coefficient of the $t$-expansion for Drinfeld modular forms, and the Fourier coefficient with respect to the trivial ideal for harmonic cochains, we have two pairings (see Sections~\ref{ssec-pairingDMF} and \ref{ssec-pairingHC}):
\begin{align*}
&\pairing_{\DMF} : \TT(\nn,\Cinf) \times M_{2,1}^{0,0}(\nn)\longrightarrow \Cinf, \\
& \pairing_{\HC,\Z} : \TT(\nn,\Z) \times \Harm_{0}(\nn,\Z)  \longrightarrow  \Z.
\end{align*}

They appear to be unrelated because there is no well-understood connection between \linebreak $t$-expansions of Drinfeld modular forms and Fourier expansions of harmonic cochains. 
However since both $\TT(\nn,\Z)$ and $\Harm_0(\nn,\Z)$ are finitely generated free $\Z$-modules of the same rank, the spaces involved in these pairings have compatible dimensions. It is natural to ask whether they are perfect.

\smallskip
Let us highlight previously known results on them and significant differences with the classical $q$-pairing. For Drinfeld modular forms, we could prove in \cite{ArmCoeffDMF} the perfectness of $\pairing_{\DMF}$ in the special case when $\nn$ is a prime of degree~$3$ but there is no known formula close to~\eqref{eq-a1TnF} in general. More intricate expressions involving $A$-linear combinations of Hecke operators were obtained in \cite{ArmCoeffDMF} (see Section~\ref{ssec-pairingDMF}). However in general they are not sufficient to get the perfectness of $\pairing_{\DMF}$ since only a small proportion of coefficients of the $t$-expansion is obtained in this way. 

For harmonic cochains, Gekeler \cite{GekAnalytical} gave a formula analogous to~\eqref{eq-a1TnF} which involves a power of $q$. He showed that the pairing $\pairing_{\HC,\Z}$  has zero left and right-kernel and becomes perfect after tensoring with $\Z[p^{-1}]$. More recently Papikian and Wei \cite{PapWei-JNT2016} proved that $\pairing_{\HC,\Z}$ is perfect over~$\Z$ when $\deg\nn=3$. See Section~\ref{ssec-pairingHC} for a review.

\smallskip
Our main result shows that, in a quite general setting, these pairings behave differently from the classical $q$-pairing.
\begin{theo}[Corollary \ref{coro-main} in the text]\label{th-main-intro}
	Suppose that the ideal $\nn$ is prime of degree~$\geq 5$. The pairings $\pairing_{\DMF}$ and $\pairing_{\HC,\Z}$ are not perfect.
\end{theo}
As a consequence, the determinant of $\pairing_{\HC,\Z}$, with respect to any $\Z$-bases of $\TT(\nn,\Z)$ and $\Harm_{0}(\nn,\Z)$, is a positive power of~$p$. Indeed we will prove that the reduction of $\pairing_{\HC,\Z}$ modulo $p$, denoted by $\pairing_{\HC,\Fp}$, is not perfect.

\smallskip
Interestingly in the proof of Theorem~\ref{th-main-intro} a key-role is played by the same element of the Hecke algebra, namely $\sum_{\deg \pp \leq 1} T_\pp$. We proceed in two steps.

First we show that this element is in the left-kernel of $\pairing_{\DMF}$ and $\pairing_{\HC,\Fp}$. For $\pairing_{\DMF}$, it has been already proven in \cite{ArmCoeffDMF}. For $\pairing_{\HC,\Fp}$, this is Proposition~\ref{prop-eltnoyaucochaine}: as far as we know, there were no such kernel elements for harmonic cochains modulo~$p$ in the previous literature.

Then we show that $\sum_{\deg \pp \leq 1} T_\pp$ is non-zero in the Hecke algebra $\TT(\nn,\Z)\otimes_{\Z} \Fp$, i.e. as an endomorphism of the space of modular forms. This is Theorem~\ref{th-main}, which proves Conjecture~6.9 formulated in \cite{ArmCoeffDMF}.  To do this we use the formalism of modular symbols over $\Fq(T)$ developed by Teitelbaum \cite{TeiMS}. First we prove a canonical isomorphism between spaces of modular forms and modular symbols modulo~$p$ which is compatible with Hecke operators (Lemma~\ref{lemm-isomcarp}). It transfers our problem to showing that $\sum_{\deg \pp \leq 1} T_\pp$ is a non-zero endomorphism on this space of modular symbols. Namely we only have to show that, when this element of the Hecke algebra is evaluated at some modular symbol, it is non-zero. Such a modular symbol is chosen and this evaluation is done in Proposition~\ref{prop-main}. To prove it we make an essential use of Teitelbaum's finite presentation of modular symbols, which is the function field counterpart of Manin's presentation. This presentation provides a way to treat the problem almost independently of the level $\nn$. The proof of Proposition~\ref{prop-main} is then of a combinatorial nature and relies on cancellation of terms modulo~$p$. To conclude, we have to show that the result of the evaluation from the proposition is non-zero: this follows from the main result of \cite{ArmBaseSM} which gives a large explicit family of linearly independent modular symbols among the generators of the presentation (such a result has no known analog for classical modular symbols). This is the reason for the assumption on $\nn$ in Theorem~\ref{th-main-intro}.
\smallskip

Theorem~1.5 in \cite{ArmTorsion} is a result on rational points of Drinfeld modular curves, based on an adaptation of Mazur's formal immersion method. Its main assumption is the perfectness of a variant of the pairing $\pairing_{\DMF}$ which is constructed from the winding element modulo~$p$. Theorem~\ref{th-main-intro} highlights very significative differences when trying to adapt Mazur's method to the Drinfeld setting. However it is not enough to show that this assumption is not satisfied: to decide on this hypothesis would require more precise knowledge on the left-kernel of $\pairing_{\DMF}$, starting with a generating family.

\smallskip
Other elements in the left-kernel of the pairings $\pairing_{\DMF}$ and $\pairing_{\HC,\Fp}$ are given in Theorem~\ref{thm-coeffDMF}, Corollary~\ref{coro-coeffDMF} and Proposition~\ref{prop-eltnoyaucochaine}. They involve Hecke operators with index of degree~$\geq 2$: one may wonder whether these elements also are non-zero endomorphisms. However the proof of Proposition~\ref{prop-main} seems technically delicate to extend to higher degrees: it relies on a general explicit formula from \cite{ArmBaseSM} for Hecke operators in terms of generators of Teitelbaum's presentation which becomes more difficult to handle here as soon as the degree is higher than one (see Proposition~\ref{prop-formuleMerel} and Remark~\ref{rema-TPdeg2}). In the last section, we gather computational data on these left-kernel elements with higher degree index. They suggest that Theorem~\ref{th-main} cannot be extended to prime levels~$\nn$ of degree~$4$. When $\nn$ is prime, our examples also suggest that the left-kernel of $\pairing_{\HC,\Fp}$ has dimension at least~$2$ if $\deg \nn \geq 7$, and the left-kernel of $\pairing_{\DMF}$ has dimension at least~$3$ if $\deg \nn \geq 5$. Overall we see this paper as a first step in understanding the behavior of these pairings.

\smallskip
The text is organized as follows. After setting some notations in Section~\ref{sec-notations}, we recall results on Drinfeld modular forms and harmonic cochains, their $t$-expansion, Fourier expansion, and Hecke operators in Section~\ref{sec-prelimfmod}. We also state comparison isomorphisms of Gekeler--Reversat between them. In Section~\ref{sec-pairings}, we introduce the pairings $\pairing_{\DMF}$, $\pairing_{\HC,\Z}$ and $\pairing_{\HC,\Fp}$ and exhibit elements in their left-kernels. Section~\ref{sec-MS} is devoted to the theory of modular symbols over $K$: its new contribution is Lemma~\ref{lemm-isomcarp}.
In Section~\ref{section-nonzeroelts} we prove  Theorem~\ref{th-main} on the element $\sum_{\deg \pp \leq 1}  T_\pp$ using computations with modular symbols, and Theorem~\ref{th-main-intro} as its corollary. Finally in Section~\ref{sec-expdata} we present computational data that go beyond Theorem~\ref{th-main-intro} and conclude with open questions.

\subsection*{Acknowledgments}
The author would like to thank M. Papikian and F.-T. Wei for discussions during the conference “New developments in the theory of modular forms over function ﬁelds” in Pisa, 2018, as well as the organizers of this event. In particular she would like to thank F.-T. Wei for suggesting to extend the study to the pairing $\pairing_{\HC,\Fp}$.

\smallskip
This work is distributed under a Creative Commons Attribution — 4.0
International licence (CC-BY 4.0): \url{https://creativecommons.org/licenses/by/4.0/}.
This research was funded, in whole or in part, by ANR -- France (French National Research Agency), PadLEfAn project ANR-22-CE40-0013. A CC-BY public copyright license has been applied by the author to the present document and will be applied to all subsequent versions up to the Author Accepted Manuscript arising from this preliminary version, in accordance with the grant's open access conditions.

\section{Notations}\label{sec-notations}

Let $\Fq$ be a finite field with $q$ elements, where $q$ is a power of a prime number $p$. Let $A = \Fq[T]$, the ring of polynomials in $T$ with coefficients in $\Fq$. We denote by $A_{+}$ the set of monic polynomials in $A$ and for any non-negative integer $d$, $A_{+\leq d}$ (resp. $A_{+d}$) the subset of polynomials of degree at most $d$ (resp. of degree $d$).

For convenience, we will call an \emph{ideal of $A$} any non-zero ideal $\nn$ and denote by $\deg \nn$ the degree of its monic generator. The \emph{level-$\nn$ Hecke congruence subgroup} of $\GL_2(A)$ is:
$$\Gamma_0(\nn) = \left\{ \begin{pmatrix}a&b\\c&d\end{pmatrix} \in \GL_2(A) \grandmid  c\equiv 0 \bmod \nn \right\}.$$
The group $\GL_2(K)$ acts from the left on the projective line $\PP^{1}(K)$ by fractional linear transformations. The \emph{cusps of $\Gamma_0(\nn)$} are the orbits under the subgroup $\Gamma_0(\nn)$ and we set $\cusps(\nn)=\Gamma_0(\nn) \bs \PP^{1}(K)$.

Let $K =\Fq(T)$ be the rational function field in $T$ i.e. the fraction field of $A$. Put $\pi = 1/T$ and let $\infty$ be the place of $K$ corresponding to $\pi$, with the corresponding absolute value $|\cdot |$, normalized such that for any $a\in A$, $|a|=q^{\deg a}$. The completion of $K$ with respect to $|\cdot|$ is the field $\Kinf = \Fq((\pi))$ with ring of integers $\Oinf = \Fq[[\pi]]$. Let $\Cinf$ be the completion of an algebraic closure of $\Kinf$.

To define the Hecke operators $T_\pp$ acting on various objects attached to $\Gamma_0(\nn)$, we will need a set of matrices in $M_2(A)$ for any ideal $\pp$ of $A$:
$$ S_{\pp} = \left\{ \begin{pmatrix} a & b \\ c & d \end{pmatrix} \grandmid (a,d) \in A_{+}\times A_{+}, (ad) = \pp, (a)+\nn = A, (b,d) \in A\times A, \deg b < \deg d\right\}.$$

\section{\texorpdfstring{Preliminaries on modular forms over $\Fq(T)$}{Preliminaries on modular forms over FqT}}\label{sec-prelimfmod}

\subsection{Drinfeld modular forms}\label{sec-DMF}
We recall basic properties on Drinfeld modular forms, their $t$-expansions and Hecke operators. For further details, see \cite{GekLNM,GekCoeffDMF,GR}.

\subsubsection{\texorpdfstring{Drinfeld modular forms and their $t$-expansion}{Drinfeld modular forms and their t-expansion}}\label{subsec-DMFexp}
Let $\Omega = \Cinf - \Kinf$, the \emph{Drinfeld half plane}. It is a rigid analytic space of dimension~$1$ over $\Cinf$ equipped with the left action of $\GL_2(\Kinf)$ given by fractional linear transformations. 
The affine \emph{Drinfeld modular curve} $Y_0(\nn)$ attached to $\Gamma_0(\nn)$ is an algebraic curve over~$K$ whose set of $\Cinf$-points is $Y_0(\nn)(\Cinf)=\Gamma_0(\nn) \bs \Omega$. It has a natural compactification $X_{0}(\nn)$ over~$K$ with $X_{0}(\nn) = Y_{0}(\nn) \sqcup \cusps(\nn)$ and $X_{0}(\nn) (\Cinf)= \Gamma_0(\nn) \bs (\Omega \sqcup \PP^{1}(K))$. The genus of $X_0(\nn)$ is denoted by $g(\nn)$ and the number of cusps by $c(\nn)$. Explicit formulas for $g(\nn)$ and $c(\nn)$ are given in \cite[Thm.~2.17]{GN} and \cite[Prop.~6.7]{GekInvDMC}.

\begin{defi}
Let $\nn \lhd A$ and $k, m$ be non-negative integers with $0 \leq m \leq q-2$ ($m$ represents a congruence class in $\Z/(q-1)\Z$). A \emph{Drinfeld modular form of weight~$k$ and type $m$ for $\Gamma_0(\nn)$} is a rigid holomorphic function $f : \Omega \to \Cinf$ such that:
\begin{enumerate}
	\item\label{item-defdmf-1} for all $\gamma = \begin{pmatrix}a&b\\c&d\end{pmatrix} \in \Gamma_0(\nn)$ and for all $z\in\Omega$, we have
	$$ f\left( \frac{az+b}{cz+d}\right) = (\det \gamma)^{-m} (cz+d)^k f(z),$$
	\item\label{item-defdmf-2} $f$ is holomorphic at all cusps of $\Gamma_0(\nn)$.
\end{enumerate}
\end{defi}
For Condition~\ref{item-defdmf-2}, see \cite[V.3]{GekLNM} and \cite[Section~2]{GR}. At the cusp infinity, it may be stated as follows. For $z\in\Omega$, let
$$ t(z) = \frac{1}{\picarlitz} \sum_{a \in A} \frac{1}{z-a}$$
where $\picarlitz$ is the period of the Carlitz exponential function. The map $t$ is holomorphic on $\Omega$ and satisfies for all $a\in A$, $t(z+a)=t(z)$. It is a uniformizer at the cusp infinity on $X_0(\nn)$. Any rigid holomorphic function $f:\Omega \to \Cinf$ satisfying Condition~\ref{item-defdmf-1} has an expansion of the form $ f(z) = \sum_{n\in\Z} a_n(f) t(z)^n$, valid for all $z\in\Omega$ such that $|t(z)|$ is small enough. The coefficients $(a_n(f))_{n\in \Z}$ are in $\Cinf$ and uniquely determine $f$. At the cusp infinity, Condition~\ref{item-defdmf-2} then means that the coefficients $(a_n(f))_{n<0}$ are zero, so the expansion has the form $$f(z) =   \sum_{n\geq 0} a_n(f) t(z)^n.$$
This is the \emph{$t$-expansion} of the Drinfeld modular form $f$. For simplicity we will omit the condition that $|t(z)|$ is small enough for this equality to hold.

\smallskip The weight $k$, the type $m$, and the coefficients of the $t$-expansion are related by elementary conditions coming from the matrices $\smallpmatrix{\lambda}{0}{0}{\lambda}$ and $\smallpmatrix{\lambda}{0}{0}{1}$ in $\Gamma_0(\nn)$, for all  $\lambda \in \Fq\et$.
Condition~\ref{item-defdmf-1} with the first set of matrices implies that $f$ is identically zero unless $k\equiv 2m \bmod(q-1)$. Moreover since $t(\lambda z)=  \lambda^{-1}t(z)$, Condition~\ref{item-defdmf-1} with the second set of matrices implies that $a_n(f) = 0$ unless $n\equiv m \bmod (q-1)$. For any $i\geq 0$, we put $b_i(f) = a_{m+(q-1)i}(f)$. With this renormalization, the $t$-expansion of $f$ may now be written as
\begin{equation}\label{eq-texpdmf}
f =  \sum_{i\geq 0} b_i(f)\, t^{m+(q-1)i}.\end{equation}

\smallskip  Let $M_{k,m}(\nn)$ be the $\Cinf$-vector space of Drinfeld modular forms of weight $k$ and type $m$ for $\Gamma_0(\nn)$. It has finite dimension, see \cite[V.6]{GekLNM}. When $\Gamma_0(\nn)=\GL_2(A)$, we will denote it by $M_{k,m}(\GL_2(A))$.

If $a_0(f) = 0 = a_1(f)$ and if similar conditions hold at the other cusps, $f$ is said to be \emph{doubly cuspidal} (Goss had observed that this condition plays a role similar to classical cusp forms). Let $M_{k,m}^{0,0}(\nn)$ be the subspace of doubly cuspidal Drinfeld modular forms in $M_{k,m}(\nn)$. When $(k,m)=(2,1)$, $M_{2,1}^{0,0}(\nn)$ is isomorphic to the space of holomorphic differential forms on the Drinfeld modular curve $X_0(\nn)$. In particular we have $\dim_{\Cinf} M_{2,1}^{0,0}(\nn)=g(\nn)$.

\begin{rema}\label{rema-DMF}
The following observations are derived immediately from the $t$-expansion~\eqref{eq-texpdmf}.
\begin{itemize}
	\item Any Drinfeld modular form of type $m>1$ is doubly cuspidal.
	\item If $f$ is doubly cuspidal and $m\in\{0,1\}$, the coefficient $b_0(f)$ is zero so its $t$-expansion starts with the term $b_1(f)t^{m+(q-1)}$.
\end{itemize}
\end{rema}

\subsubsection{Hecke operators on Drinfeld modular forms} They are defined using the left-action on the Drinfeld half plane $\Omega$ by the matrices in $S_\pp$.

\begin{defi}
	Let $\pp\lhd A$ with monic generator $P\in A$. For $f\in M_{k,m}(\nn)$, let
$$	\forall z \in \Omega, \qquad (T_\pp f)(z) = \frac{1}{P} \sum_{\smallpmatrix{a}{b}{0}{d} \in S_\pp} a^k f \left( \frac{az+b}{d} \right).$$
\end{defi}
This defines a $\Cinf$-linear transformation $T_\pp$ of the space $M_{k,m}(\nn)$, called the \emph{$\pp$-th Hecke operator}. We also denote it by $T_P$.
\begin{rema}
	We follow the normalization of \cite[4.3]{ArmCoeffDMF} which differs from other references. A more standard choice would be $P^{m+1-k}T_\pp$: it coincides with our definition of $T_\pp$ when $(k,m)=(2,1)$, which will be the setting for most of this text.
\end{rema}

We recall usual properties of the Hecke operators $T_\pp$ acting on Drinfeld modular forms:
\begin{itemize}
	\item for \emph{any} ideals $\pp,\pp'$ of $A$, they satisfy $T_\pp T_{\pp'} =T_{\pp \pp'} =T_{\pp'} T_\pp$,
	\item they stabilize the subspace $M_{k,m}^{0,0}(\nn)$.
	\end{itemize}
The first property distinguishes them from Hecke operators acting on classical modular forms, or on harmonic cochains in characteristic zero (see Section~\ref{sssection-heckeopharm}).

\begin{defi}
	Let $\TT_{k,m}(\nn,\Cinf)$ be the $\Cinf$-subalgebra of $\End_{\Cinf}(M_{k,m}^{0,0}(\nn))$ generated by the Hecke operators $(T_\pp)_{\pp \lhd A}$. This commutative algebra is called the \emph{Hecke algebra for doubly cuspidal Drinfeld modular forms}. If $(k,m)=(2,1)$, we will simply denote it by $\TT(\nn,\Cinf)$.
\end{defi}

\subsection{Harmonic cochains}\label{sec-HC}
We recall basic properties on harmonic cochains for the Bruhat-Tits tree of $\PGL_2(\Kinf)$, their Fourier expansion and Hecke operators. For further details we refer to \cite{GekAnalytical,GekImproper,GR}.

\subsubsection{\texorpdfstring{The Bruhat-Tits tree of $\PGL_2(\Kinf)$}{The Bruhat-Tits tree of PGL2(Kinf)}}\label{subs-BTtree}

\begin{defi}
The \emph{Bruhat-Tits tree $\BT$ attached to $\PGL_2(\Kinf)$} is the tree with set of vertices $V(\BT) = \GL_2(\Kinf)/\Kinf\et \, \GL_2(\Oinf)$ and set of oriented edges $E(\BT) = \GL_2(\Kinf)/\Kinf\et\, \Iwa$ where $\Iwa$ is the Iwahori subgroup
$$\Iwa = \left\{ \begin{pmatrix} a&b\\c&d\end{pmatrix} \in \GL_2(\Oinf) \grandmid c \equiv 0 \bmod \pi \right\}.$$
\end{defi}
It is a $(q+1)$-regular tree. For any edge $e \in E(\BT)$, let $o(e) \in V(\BT)$ be its origin and $\overline{e} \in E(\BT)$ be the opposite edge. If the edge $e$ is represented by a matrix $g\in \GL_2(\Kinf)$,
its origin $o(e)$ is represented by $g$ and the opposite edge $\overline{e}$ is represented by $g\smallpmatrix{0}{1}{\pi}{0}$.
A set of representatives for the edges up to orientation is (see \cite[Section~1]{GekImproper}):
$$ \begin{pmatrix} \pi^k & u \\ 0 & 1 \end{pmatrix} \quad(k\in\Z, u \in\Kinf/ \pi^k \Oinf).$$

The arithmetic group $\Gamma_0(\nn)$ acts from the left on $\GL_2(\Kinf)$ without inversion. It provides a quotient graph $\Gamma_0(\nn)\bs\BT$ whose set of edges is $\Gamma_0(\nn)\bs V(\BT)$ and set of oriented edges is $\Gamma_0(\nn)\bs E(\BT)$. The structure of the graph $\Gamma_0(\nn)\bs \BT$ is well-understood: it is the edge-disjoint union of a finite subgraph with a finite number of half-lines, called the \emph{cusps of $\Gamma_0(\nn)\bs\BT$}, which are in bijection with the cusps of $\Gamma_0(\nn)$ (see \cite[II.2.3 thm. 9 and II.2.8 p. 172]{SerreTrees}, \cite{GN}). Also the genus of the graph $\Gamma_0(\nn)\bs \BT$ coincides with the genus $g(\nn)$ of $X_0(\nn)$.
\subsubsection{Harmonic cochains}\label{sssection-hcoch} Let $R$ be a commutative ring with unity.

\begin{defi}An $R$-valued \emph{harmonic cochain} on $\BT$ is a map $F : E(\BT) \to R$ satisfying:
\begin{enumerate}
	\item for all $e \in E(\BT)$, $F(e)+F(\overline{e}) =0$,
	\item for all $v\in V(\BT)$, $\displaystyle\sum_{e \in E(\BT),\, o(e)=v} F(e)=0$.
\end{enumerate}
We say that $F$ is \emph{$\Gamma_0(\nn)$-invariant} if, for all $\gamma \in \Gamma_0(\nn)$ and $e\in E(\BT)$, we have $F(\gamma e) = F(e)$. When $F$ is $\Gamma_0(\nn)$-invariant, we say that $F$ is \emph{cuspidal} if it  is finitely supported when viewed as a function on the edges of the graph $\Gamma_0(\nn)\bs\BT$.
\end{defi}
Let $\Harm(\nn,R)$ be the group of $R$-valued harmonic cochains on $\BT$ which are $\Gamma_0(\nn)$-invariant, and $\Harm_{0}(\nn,R)$ the subgroup consisting of the cuspidal ones. The canonical map
$$ \Harm_{0}(\nn,\Z) \otimes_{\Z} R \to \Harm_{0}(\nn,R)$$
is injective but not surjective in general. Let $\Harm_{0,0}(\nn,R)$ be its image in $\Harm_{0}(\nn,R)$ (see \cite[3.6.2]{GR}). By definition, we have $\Harm_{0,0}(\nn,R) \simeq \Harm_{0}(\nn,\Z)\otimes_{\Z} R$ for any ring $R$ and this construction commutes with arbitrary ring extensions $R'/R$. However it is clear that the formation of $\Harm(\nn,R)$ and $\Harm_{0}(\nn,R)$ only commutes with flat ring extensions $R'/R$ (for instance $\Z \subset \Q \subset \C$). Moreover if $R$ is flat over $\Z$, we have $\Harm_{0}(\nn,R)=\Harm_{0,0}(\nn,R)$.

\begin{rema}When $R=\C$, Drinfeld has shown that such harmonic cochains, with additional assumptions, have an interpretation as automorphic forms for $\GL_2(K)$ (see  \cite[Section~4]{GN} for a precise formulation).
\end{rema}

When $R=\Z$, it is known from \cite[3.2]{GR} that $\Harm(\nn,\Z)$ and $\Harm_{0}(\nn,\Z)$ are finitely generated free $\Z$-modules with 
$$\rk_{\Z} \Harm	(\nn,\Z) = g(\nn)+c(\nn)-1\quad\text{and}\quad \rk_{\Z} \Harm_0(\nn,\Z) = g(\nn).$$

\subsubsection{Fourier expansion} Our references are \cite{GekAnalytical,GekImproper}. Although we will not need it, most results of this subsection and the next one may be extended to harmonic cochains with values in a \emph{coefficient ring}, namely 
a ring in which $p$ is invertible and which is a quotient of a discrete valuation ring which contains $p$-th roots of unity (see \cite[2.3]{PapWei-Doc2015}).

Any harmonic cochain $F \in \Harm_{0}(\nn,\C)$ may be viewed as a  $\C$-valued function on $\GL_2(\Kinf)$ which is invariant under the subgroup $\{ \smallpmatrix{a}{b}{c}{d} \in \GL_2(A) \mid c= 0 \}$. Let $\eta_0 : \Fp \to \C\et$ be a chosen non-trivial additive character of $\Fp$. Let $\tr : \Fq\to\Fp$ be the trace and $\eta : \Kinf \to \C\et$ be the additive character defined by 
$$ \eta\left(\sum_{k\in\Z} a_k \pi^k\right)  = \eta_0 (\tr(a_1)).$$

\begin{defi}[{\cite[(3.12)(iii)]{GekAnalytical}}] For any $\mm\lhd A$ with monic generator $M \in A$, the \emph{Fourier coefficient of $F$ with respect to $\mm$} is
\begin{equation}\label{eq-coch-Fcoeff}c_{\mm}(F) = q^{-1-\deg \mm} \sum_{u \in \pi \Oinf / \pi^{2+\deg \mm}\Oinf} F\left( \begin{pmatrix} \pi^{2+\deg \mm}&u\\0&1\end{pmatrix}\right) \eta(-Mu).
\end{equation}
When $\mm=A$ is the trivial ideal, we denote it by $c_{1}(F)$.
\end{defi}

We collect basic properties of the Fourier expansion.
\begin{prop}[Gekeler]\label{prop-gekDvtfour}Let $F \in \Harm_{0}(\nn,\C)$.
	\begin{enumerate}
		\item\label{item-coch-Fexp} We have 
		$$
			\forall k\in\Z, \, \forall u \in \Kinf, \quad F\left( \begin{pmatrix}\pi^k & u \\ 0 & 1 \end{pmatrix}\right) = \sum_{0 \leq j  \leq k-2} q^{-k+2+j} \sum_{\deg \mm =j} c_{\mm}(F) \nu(Mu)
		$$
		where $\nu(y)$ is equal to $-1$ if $y$ has a term of order~$\pi$ in its $\pi$-expansion, and $q-1$ otherwise. In particular the Fourier coefficients $(c_{\mm}(F))_{\mm\lhd A}$ uniquely determine $F$ in $\Harm_0(\nn,\C)$.
		\item If $F \in \Harm_{0}(\nn,\Z)$, its Fourier coefficients $(c_{\mm}(F))_{\mm\lhd A}$ belong to $\Z[p^{-1}]$.
		\item\label{item-CF1} The Fourier coefficient with respect to the trivial ideal is
		$$
			c_{1}(F) = -F \left( \begin{pmatrix}\pi^2 & \pi \\0 & 1 \end{pmatrix}\right).
		$$
	\end{enumerate}
\end{prop}
\begin{proof}
See (3.12)(i'), Corollary~3.15 and (3.16) in \cite{GekAnalytical}, 
\end{proof}

\subsubsection{Hecke operators on harmonic cochains}\label{sssection-heckeopharm}  They are defined using the left-action on $\GL_2(\Kinf)$ by the matrices in $S_\pp$. Our references are \cite[4.9]{GR} and \cite[1.10]{GekAnalytical}.

\begin{defi}
Let $\pp$ be an ideal of $A$ with monic generator $P \in A$. For $F\in \Harm_{0}(\nn,R)$, let
$$	\forall e \in E(\BT), \qquad (T_\pp F)(e) = \sum_{\smallpmatrix{a}{b}{0}{d} \in S_\pp} F \left( \begin{pmatrix} a & b \\ 0 & d \end{pmatrix} e \right).$$	
\end{defi}
This defines an $R$-linear transformation $T_\pp$ of the space $\Harm_0(\nn,R)$, called the \emph{$\pp$-th Hecke operator}. We also denote it by $T_P$. 

The Hecke operators on cuspidal harmonic cochains stabilize $\Harm_{0}(\nn,R)$ and $\Harm_{0,0}(\nn,R)$. They satisfy:
\begin{itemize}
  \item for any \emph{coprime} ideals $\pp,\pp'$ of $A$, $T_{\pp \pp'} = T_{\pp} T_{\pp'}$,
\item for any prime ideal $\qq \nmid \nn$ and $i\geq 2$, $T_{\qq^{i}} = T_{\qq^{i-1}} T_\qq -q^{\deg \qq} T_{\qq^{i-2}}$, 
\item  for any prime ideal $\qq\mid \nn$ and $i\geq0$, $T_{\qq^i} = T_\qq^{i}$.
\end{itemize}
For $e \in E(\BT)$, let $n(e) = \# (\Gamma_0(\nn)_e / \Fq\et)$ where $\Gamma_0(\nn)_e$ is the stabilizer of $e$ under the action of $\Gamma_0(\nn)$.
Recall that the $\C$-vector space $\Harm_{0}(\nn,\C)$ is equipped with the Hermitian inner product, called the \emph{Petersson product}, defined by
$$\forall (F,G) \in \Harm_{0}(\nn,\C)\times \Harm_{0}(\nn,\C), \quad \langle F,G \rangle = \sum_{e \in  E(\Gamma_0(\nn) \bs \BT)} \frac{1}{2n(e)} F(e) \overline{G(e)}.$$
If $\pp$ and $\nn$ are coprime, the Hecke operator $T_\pp$ is self-adjoint with respect to the Petersson product.

\begin{defi}
	Let $\TT(\nn,\Z)$ be the $\Z$-subalgebra of $\End_{\Z}(\Harm_{0}(\nn,\Z))$ generated by the Hecke operators $(T_\pp)_{\pp \lhd A}$. This commutative algebra is called the \emph{Hecke algebra for cuspidal harmonic cochains}.
\end{defi}

We recall the action of Hecke operators on the Fourier coefficient with respect to the trivial ideal (see also Proposition~\ref{prop-connupairingharm}).
\begin{prop}[{\cite[3.13, 3.17]{GekAnalytical}}]\label{prop-C1TMF}
\begin{enumerate}
	\item 	For any $F\in \Harm_{0}(\nn,\C)$ we have 
	\begin{equation}\label{eq-C1TPF}
\forall \pp \lhd A,\quad c_{1}(T_\pp F) = q^{\deg \pp} c_{\pp}(F).
		\end{equation}
	\item The pairing 
	$$\begin{array}{ccl}
		\TT(\nn,\Z) \otimes_{\Z} \Q \times \Harm_0(\nn,\Q) & \longrightarrow & \Q \\
		(\theta,F) & \longmapsto & c_1(\theta F)
	\end{array}$$
	is perfect over $\Q$.
\end{enumerate}
\end{prop}

The structure of the Hecke algebra as a $\Z$-module is well-known.
\begin{lemm}\label{lem-rankHeckeZ}
The Hecke algebra $\TT(\nn,\Z)$ is a free $\Z$-module of rank $g(\nn)$.
\end{lemm}
\begin{proof}
The Hecke algebra  $\TT(\nn,\Z)$ is a submodule of $\End_{\Z}(\Harm_0(\nn,\Z))$ where $\Harm_{0}(\nn,\Z)$ is free of finite rank over $\Z$. Therefore the $\Z$-module $\TT(\nn,\Z)$ is free of finite rank. The perfect $\Q$-pairing from Proposition~\ref{prop-C1TMF} gives $ \dim_{\Q} \TT(\nn,\Z)\otimes_{\Z}\Q = \dim_{\Q} \Harm_0(\nn,\Q)=g(\nn)$. We conclude that $\rk_{\Z} \TT(\nn,\Z)=g(n)$.
\end{proof}

\subsection{Comparison isomorphisms}\label{ssection-comparisonisom}

Following Teitelbaum \cite{TeitPoisson} and Gekeler--Reversat \cite{GR}, there are canonical isomorphisms between certain $\Cinf$-valued harmonic cochains and Drinfeld modular forms of weight $2$ and type $1$. We recall the main statements here. Our reference is \cite[Sections~1, 5, 6]{GR}.

The \emph{building map}  $\lambda : \Omega \to \BT(\R)$ is a $\GL_2(\Kinf)$-equivariant canonical map where $\BT(\R)$ is the real realization of the Bruhat-Tits tree $\BT$. For any edge $e \in Y(\BT)$, the inverse image $\lambda^{-1}(e)$ is an oriented annulus isomorphic with $\{ z \in \Cinf \mid |\pi| \leq |z| \leq 1 \}$ in $\Cinf$.

Let $f\in M_{2,1}^{0,0}(\nn)$ be a doubly cuspidal Drinfeld modular form of weight $2$ and type~$1$. Then $f(z)dz$ is a holomorphic differential on $\Omega$ so  for any edge $e\in Y(\BT)$ it has a residue in the oriented annulus $\lambda^{-1}(e)$, denoted by $\res_{e}(f(z)dz)$. This provides the \emph{residue map}
$$\begin{array}{rrcl}
\res\, f: & E(\BT) &\longrightarrow &\Cinf  \\
&e &\longmapsto &\res_{e} (f(z)dz)
\end{array}	$$
which turns out to be a $\Cinf$-valued $\Gamma_0(\nn)$-invariant harmonic cochain.

\begin{defi}
Let $M_{2,1}^{0,0}(\nn,\Fp)$ be the subspace of $M_{2,1}^{0,0}(\nn,\Cinf)$ consisting of Drinfeld modular forms who have all their residues in $\Fp$.\end{defi}

This subspace is stable under all Hecke operators. Let $\TT(\nn,\Fp)$ be the commutative $\Fp$-subalgebra of $\End_{\Fp}(M_{2,1}^{0,0}(\nn,\Fp))$ generated by $(T_\pp)_{\pp \lhd A}$.

\begin{theo}[Teitelbaum, Gekeler--Reversat -- see {\cite[3.6.1 and 6.5.3]{GR}}]\label{theo-GRisom} 
The residue map induces isomorphisms of vector spaces over $\Cinf$ (resp. $\Fp$):
\begin{align*}
&\res_{\Cinf} : M_{2,1}^{0,0}(\nn) \xrightarrow{\sim} \Harm_{0,0}(\nn,\Cinf) \\
&\res_{\Fp} : M_{2,1}^{0,0}(\nn,\Fp) \xrightarrow{\sim} \Harm_{0,0}(\nn,\Fp)
\end{align*}
which are $\GL_2(\Kinf)$-equivariant and compatible with Hecke operators.
\end{theo}	

This is illustrated by the following commutative diagram in which the vertical maps are injective and come from $\Fp \hookrightarrow \Cinf$:
\begin{equation}\label{eq-diagcomm}
\begin{tikzcd}
M_{2,1}^{0,0}(\nn,\Fp) \arrow[r, "\res_{\Fp}", "\simeq"'] \arrow[d,]
& \Harm_{0,0}(\nn,\Fp) \arrow[d] \\
M_{2,1}^{0,0}(\nn) \arrow[r, "\res_{\Cinf}","\simeq"' ]
& \Harm_{0,0}(\nn,\Cinf).
\end{tikzcd}
\end{equation}
In particular, $M_{2,1}^{0,0}(\nn,\Fp)$ defines a canonical $\Fp$-structure on the $\Cinf$-space of Drinfeld modular forms $M_{2,1}^{0,0}(\nn)$. 
From Theorem~\ref{theo-GRisom} and Lemma~\ref{lem-rankHeckeZ}, we get:
\begin{coro}\label{cor-isomalghecke}We have canonical isomorphisms as vector spaces over $\Cinf$ (resp.~$\Fp$):
$$\TT(\nn,\Cinf) \simeq \TT(\nn,\Z)\otimes_{\Z} \Cinf \quad\text{and}\quad \TT(\nn,\Fp) \simeq \TT(\nn,\Z)\otimes_{\Z} \Fp$$
and they have dimension $g(\nn)$.
\end{coro}	
We will now use these as identifications and simply write $\TT (\nn,\Cinf)$ for the first Hecke algebra and $\TT(\nn,\Fp)$ for the second one.	

\section{Pairings between modular forms and Hecke algebra}\label{sec-pairings}

\subsection{Pairing for Drinfeld modular forms}\label{ssec-pairingDMF}

We consider doubly cuspidal Drinfeld modular forms of weight $k$ with type $m \in \{0,1\}$, so that the first coefficient of their $t$-expansion is $b_1$ (see Remark~\ref{rema-DMF}). By analogy with classical modular forms, we have the $\Cinf$-bilinear map already studied in \cite{ArmCoeffDMF}:
\begin{equation}\label{eq-pairingDMFgeneral}
	\begin{array}{ccl}
\TT_{k,m}(\nn,\Cinf) \times M_{k,m}^{0,0}(\nn) & \longrightarrow & \Cinf \\
 (\theta,f) & \longmapsto & b_1(\theta f).
\end{array}
\end{equation}

When $(k,m)=(2,1)$, we denote this pairing by 
$$\pairing_{\DMF} : \TT(\nn,\Cinf) \times M_{2,1}^{0,0}(\nn)\longrightarrow \Cinf.$$ 
Note that we have $\dim_{\Cinf} \TT(\nn,\Cinf) = g(\nn) = \dim_{\Cinf} M_{2,1}^{0,0}(\nn)$ by Corollary~\ref{cor-isomalghecke}.

\medskip
Let us recall known instances of spaces of small dimension where this pairing is perfect. They follow from considerations on the action of the Hecke operators $(T_\pp)_{\deg \pp=1}$ on the $t$-expansion.
\begin{prop}[{\cite[Thm.~7.7]{ArmCoeffDMF}}]\label{prop-pairingDMF-casparfait}
The pairing \eqref{eq-pairingDMFgeneral} is perfect for the following spaces of Drinfeld modular forms and their corresponding Hecke algebras:
\begin{enumerate}
	\item $M_{k,1}^{0,0}(\GL_2(A))$ when $k<q^2(q+1)$,
	\item $M_{2,1}^{0,0}(\nn)$ when $\nn$ is any prime of degree~$3$.
\end{enumerate}
\end{prop}

\medskip
In the opposite direction, we have previously exhibited families of elements in the left-kernel of the pairing~\eqref{eq-pairingDMFgeneral}. To state them, we need the Carlitz module $C : A \to A[x]$ which is the morphism of $\Fq$-algebras defined by $C(T) =Tx+x^q$. For all $a\in A$, the image of $a$ under $C$ is the $\Fq$-linear polynomial
$$C(a) = \sum_{k=0}^{\deg a} C_{a,k} x^{q^k}$$
with coefficients $C_{a,k} \in A$ and $C_{a,0}=a$. Moreover if $a \in A_{+}$, we have $C_{a,\deg a}=1$.

\begin{theo}[{\cite[Thm~1.1]{ArmCoeffDMF}}]\label{thm-coeffDMF} Let $m\in\{0,1\}$. The left-kernel of the pairing~\eqref{eq-pairingDMFgeneral} contains the following elements of the Hecke algebra:
\begin{enumerate}
\item $\displaystyle\sum_{P \in A_{+1}} P^{1-m} T_P + T_1$,
\item $\displaystyle\sum_{P\in A_{+d}} C_{P,0}^{i_0} \cdots C_{P,d-1}^{i_{d-1}} T_P$ for any $d\geq 1$ and any $d$-tuple $(i_0,\ldots,i_{d-1})$ of non-negative integers satisfying
\begin{align*}
&\forall j \in \{0,\ldots,d-1\}, \quad 0 \leq i_j \leq q-m \\
\text{and}\quad &i_0 +\cdots + i_{d-1} \leq (d-1)(q-1)-m.
\end{align*}
\item $\displaystyle\sum_{P \in A_{+d}} P^\ell T_P$ for any integer $\ell$ with $0 \leq \ell\leq q-m$ and $d \geq 1 + (\ell+m)/(q-1)$,
\item $\displaystyle\sum_{P \in A_{+d}} T_P$ if $d\geq 2$ or if $(d,m)=(1,0)$.
\end{enumerate}	
\end{theo}	
\begin{rema}
Some of these statements are valid without assuming $m\in\{0,1\}$ (see \cite{ArmCoeffDMF}).
\end{rema}

\begin{coro}\label{coro-coeffDMF}
	The left-kernel of the pairing $\varphi_{\DMF}$ contains the elements:
$$\sum_{P\in A_{+\leq 1}} T_P \quad\text{and}\quad \sum_{P\in A_{+2}} P^\ell T_P \text{ for any }0\leq \ell \leq q-2.$$
\end{coro}

When $(k,m)=(2,1)$, we will focus on the only element from Corollary~\ref{coro-coeffDMF} and Theorem~\ref{thm-coeffDMF} that involves Hecke operators of degree~$1$, namely $\sum_{P\in A_{+\leq 1}} T_P$. It is expected to be  non-zero as an endomorphism of Drinfeld modular forms in a quite general setting.
\begin{conj}[{\cite[Conj.~6.9]{ArmCoeffDMF}}]\label{conj-69} Suppose that $(k,m)=(2,1)$ and that the ideal $\nn$ is prime of degree~$\geq 5$. The element
	$$\displaystyle\sum_{\deg \pp \leq 1} T_\pp$$ is non-zero in $\TT(\nn,\Fp)$. Consequently the pairing $\pairing_{\DMF}$ is not perfect.
\end{conj}
Numerical evidence to support this claim was given in \cite[6.4]{ArmCoeffDMF}. The assumption $(k,m)=(2,1)$ had to do with our way of computing numerical data, namely using the comparison isomorphisms recalled in Section~\ref{ssection-comparisonisom}. It would be interesting to gather numerical data on $\sum_{\deg \pp \leq 1 } T_\pp$ for spaces of Drinfeld modular forms of higher weight and type.

\subsection{Pairings for harmonic cochains}\label{ssec-pairingHC}

Following \cite[p.~44]{GekAnalytical} and \cite[Section~2.3]{PapWei-Doc2015}, we consider the $\Z$-bilinear map attached to the Fourier coefficient with respect to the trivial ideal:
$$\begin{array}{rccl}
	\pairing_{\HC,\Z}: & \TT(\nn,\Z) \times \Harm_{0}(\nn,\Z) & \longrightarrow & \Z \\
	& (\theta ,F) & \longmapsto & c_{1}(\theta F).
\end{array}$$
It is well-defined by Proposition~\ref{prop-gekDvtfour}/\ref{item-CF1}.

\begin{prop}[Gekeler, Papikian-Wei]\label{prop-connupairingharm}\ 
\begin{enumerate}
	\item\label{item-gek1} The pairing $\pairing_{\HC,\Z}$ has zero left and right-kernel (i.e. it is non-degenerate over $\Z$).
	\item\label{item-gek2} After tensoring with $\Z[p^{-1}]$, the pairing $\pairing_{\HC,\Z}$ becomes perfect.
	\item\label{item-pw}  The pairing $\pairing_{\HC,\Z}$ is perfect over $\Z$ when $\deg \nn = 3$.
\end{enumerate}	
\end{prop}
\begin{proof}
The first two statements follow from \eqref{eq-C1TPF}, see \cite[Thm.~3.17]{GekAnalytical}. For the third one, see \cite[Thm.~4.1]{PapWei-JNT2016}.
\end{proof}

\begin{rema}Let us mention other related results on this pairing.
\begin{enumerate}
\item The first two statements of Proposition~\ref{prop-connupairingharm} remain true for harmonic cochains with values in a coefficient ring (\cite[Thm~2.20]{PapWei-Doc2015}).
\item The perfection of $\pairing_{\HC,\Z}$ after scalar extension to $\C$ remains true for harmonic cochains in $\Harm(\nn,\C)$ and their corresponding Hecke algebra (\cite[Lem.~5.1]{ArmWei}).\end{enumerate}
\end{rema}

We obtain a pairing modulo the characteristic~$p$:
$$\begin{array}{rccl}
	\pairing_{\HC,\Fp}: &  \TT(\nn,\Fp) \times \Harm_{0,0}(\nn,\Fp) & \longrightarrow & \Fp \\
	& (\theta \bmod p,F \bmod p) & \longmapsto & c_{1}(\theta F) \bmod p.
\end{array}$$
Similarly to the pairing $\pairing_{\DMF}$, we have $\dim_{\Fp} \TT(\nn,\Fp) = g(\nn) = \dim_{\Fp} \Harm_{0,0}(\nn,\Fp)$ by Corollary~\ref{cor-isomalghecke}. 
Note that \eqref{eq-C1TPF} does not implies that $T_\pp$ is in the left-kernel of $\pairing_{\HC,\Fp}$ because $c_{\pp}(F) \in \Z[p^{-1}]$ by Proposition~\ref{prop-gekDvtfour}.

\smallskip
We now give elements in the left-kernel of $\pairing_{\HC,\Fp}$.
\begin{prop}\label{prop-eltnoyaucochaine}
The left-kernel of $\pairing_{\HC,\Fp}$ contains the following elements:
\begin{enumerate}
\item for any $d\geq 1$, $ \sum_{\deg \pp \leq d} T_\pp$
\item for any $d\geq 2$, $\sum_{\deg \pp=d} T_\pp$.
\end{enumerate}
\end{prop}
\begin{proof}
\begin{enumerate}
	\item  Let $u\in \Kinf/\pi^{k} \Oinf$ and $F \in \Harm_{0}(\nn,\Z)$. 	
	Let $d\geq 1$ and put $k=d+2\geq 3$. By Proposition~\ref{prop-gekDvtfour}/\ref{item-coch-Fexp} and \eqref{eq-C1TPF}, we have
	\begin{align*}
	 q^{k-2} F \left( \begin{pmatrix}\pi^k & u\\0&1\end{pmatrix}\right)& = \sum_{0 \leq j \leq k-2} q^j\sum_{\deg \mm = j} c_{\mm}(F) \nu(Mu)\\
	 & = \sum_{0 \leq j \leq k-2} \sum_{\deg \mm = j} \nu(Mu)\, c_{1}(T_\mm F) \\
	 &	= c_{1}(\theta F)
	 \end{align*}
where $\theta = \sum_{\deg \mm \leq k-2} \nu(Mu) \, T_\mm$. Now observe that $\theta \equiv -\sum_{\deg \mm \leq k-2} T_\mm \bmod p$. Moreover  $q^{k-2} F \left( \smallpmatrix{\pi^k}{u}{0}{1} \right)\in q\Z$ since $k\geq 3$. We get $c_{1}(\theta F) \equiv 0  \bmod p$ and the conclusion follows.
\item It is a direct consequence of the first statement.
\end{enumerate} 
\end{proof}

\section{Modular symbols over $\Fq(T)$}\label{sec-MS} This theory of modular symbols over the function field $K$ has been introduced by Teitelbaum \cite{TeiMS}. We recall the main statements we will use, as well as additional results from \cite{ArmBaseSM}.

\subsection{Spaces of modular symbols}
Recall that $\GL_2(K)$ acts from the left on the projective line $\PP^{1}(K)$. Let $M=\Div^{0}(\PP^1(K))$, the $\Z$-module of degree-$0$ divisors supported on $\PP^{1}(K)$ with the corresponding left-action of $\GL_2(K)$. Let $R$ be a commutative ring with unity.

\begin{defi}
	Set $\SM(\nn,R) = H_{0}(\Gamma_0(\nn), M \otimes_{\Z} R)$, the group of coinvariants of $M\otimes_{\Z}R$ relative to the action of $\Gamma_0(\nn)$.
	It has a canonical $R$-module structure. Its elements are called the \emph{modular symbols for $\Gamma_0(\nn)$ with coefficients in $R$}.
\end{defi} 
For any $r$ and $s$ in $\PP^{1}(K)$, the divisor $(s)-(r) \in M$ defines a class $[r,s]$ in $\SM(\nn,R)$. All these classes form a generating set of $\SM(\nn,R)$ as an abelian group.

\begin{defi}
	Set $B(\nn) = \Div^{0}(\Gamma_0(\nn) \bs \PP^{1}(K))$, the $\Z$-module of degree-$0$ divisors supported on $\Gamma_0(\nn) \bs \PP^{1}(K)(=\cusps(\nn))$ and set $B(\nn, R) = B(\nn) \otimes_{\Z} R$. The boundary map 
	$$\begin{array}{rcl}
	\SM(\nn,R) & \longrightarrow & B(\nn,R) \\
	\lbrack r,s \rbrack &\longmapsto & (\Gamma_0(\nn)s) - (\Gamma_0(\nn)r)
	\end{array}$$ 
is surjective. Its kernel is denoted by $\SM_0(\nn,R)$ and its elements are called the \emph{cuspidal modular symbols for $\Gamma_0(\nn)$ with coefficients in $R$}.
\end{defi}
By construction, there are $R$-modules isomorphisms 
\begin{align}
	\SM(\nn,\Z) \otimes_{\Z} R \simeq \SM(\nn, R), & \\
	\SM_0(\nn,\Z)\otimes_{\Z} R \simeq \SM_{0}(\nn,R) & \text{ if $R/\Z$ is a flat ring extension.}
\end{align}
We also recall results on the torsion of modular symbols.
\begin{prop}[{\cite[p. 277]{TeiMS}}]\label{prop-smtorsion}
	The torsion of $\SM(\nn,\Z)$ is annihilated by $q^2-1$. When $\nn$ is prime of odd (resp. even) degree, the torsion of $\SM(\nn,\Z)$ is zero (resp. cyclic of order $q+1$).
\end{prop}

\subsection{Modular symbols as geodesics}
Let $H_{1}(\Gamma_0(\nn)\bs\BT,\cusps(\nn),\Z)$ be the first relative homology group of the graph $\Gamma_0(\nn)\bs\BT$ with respect to its cusps and $H_{1}(\Gamma_0(\nn)\bs\BT,\Z)$ its subgroup of cycles. For any $r$ and $s$ in $\Gamma_0(\nn)\bs\PP^{1}(K)$, there exists a unique geodesic $g_{r,s}$ in the Bruhat-Tits tree~$\BT$ which connects $r$ to $s$. 
By \cite[p.~277]{TeiMS}, by mapping the modular symbol $[r,s]$ to the image of $g_{r,s}$ in $\Gamma_0(\nn)\bs\BT$, we have an isomorphism of $\Z$-modules
\begin{equation}\label{eq-isomSMH1graphe} \SM_0(\nn,\Z) / \SM_0(\nn,\Z)\tors \simeq H_{1}(\Gamma_0(\nn)\bs\BT,\Z).\end{equation}
As a consequence, we have
$$\dim_\Q \SM(\nn,\Q) = g(\nn)+c(\nn)-1 \quad\text{and}\quad  \dim_\Q \SM_{0}(\nn,\Q) = g(\nn)$$
where $g(\nn)$ and $c(\nn)$ were introduced in Section~\ref{subsec-DMFexp}.

\subsection{Hecke operators on modular symbols}
They are defined using the left-action on $\PP^{1}(K)$ by the matrices in $S_\pp$.

\begin{defi}
	Let $\pp\lhd A$. For any $r$ and $s$ in $\PP^{1}(K)$, let 
	$$	T_\pp [r,s] = \sum_{g \in S_\pp}  \left[ gr,gs \right] = \sum_{\smallpmatrix{a}{b}{0}{d} \in S_\pp} \left[ \frac{ar+b}{d},\frac{as+d}{d} \right] \quad \in \SM(\nn,R).$$	
\end{defi}
It extends to a $R$-linear transformation $T_\pp$ of the space $\SM(\nn,R)$, called the \emph{$\pp$-th Hecke operator}. We also denote it by $T_P$ where $P$ is the monic generator of $\pp$. 

Their properties are similar to Hecke operators on harmonic cochains, namely they stabilize $\SM_{0}(\nn,R)$ and satisfy:
\begin{itemize}
	\item for any \emph{coprime} ideals $\pp,\pp'$ of $A$, $T_{\pp \pp'} = T_{\pp} T_{\pp'}$,
	\item for any prime ideal $\qq\nmid\nn$ and $i\geq 2$, $T_{\qq^{i}} = T_{\qq^{i-1}} T_\qq -q^{\deg \qq} T_{\qq^{i-2}}$, 
	\item  for any prime ideal $\qq\mid\nn$ and $i\geq0$, $T_{\qq^i} = T_\qq^{i}$.
\end{itemize}

\subsection{Canonical isomorphism with harmonic cochains}
We recall a canonical isomorphism between cuspidal modular symbols and harmonic cochains, by combining statements from \cite{TeiMS}, \cite{GR} and \cite{GN}.

\begin{lemm}[{\cite[Lem.~4.4]{ArmBaseSM}}]\label{lemm-isomSMcochsurZ}
	There is a canonical isomorphism of $\Z$-modules
	$$ \SM_0(\nn,\Z) / \SM_0(\nn,\Z)\tors \xrightarrow{\sim} \Harm_0(\nn,\Z)$$
	and it is compatible with Hecke operators.
\end{lemm}	
This map is obtained by composing the isomorphism \eqref{eq-isomSMH1graphe} with the canonical isomorphism $$j : H_{1}(\Gamma_0(\nn)\bs\BT,\Z) \xrightarrow{\sim} \Harm_0(\nn,\Z).$$ Indeed by \cite[3.2.5]{GR}, any cycle $\varphi\in H_{1}(\Gamma_0(\nn)\bs\BT,\Z)$ defines a harmonic cochain $j(\varphi)$ by
$$ \forall e\in Y(\BT), \quad j(\varphi)(e) = n(e) \varphi(\tilde{e})$$
where $\tilde{e}$ is the image of $e$ in $E(\Gamma_0(\nn)\bs\BT)$ and $n(e)$ was defined in
Section~\ref{sssection-heckeopharm}. This provides an injective $\Z$-linear map $j : H_{1}(\Gamma_0(\nn)\bs\BT,\Z) \to \Harm_0(\nn,\Z)$ with finite cokernel. Such a construction is still valid when $\Gamma_0(\nn)$ is replaced with an arbitrary arithmetic group $\Gamma \subset \GL_2(K)$ (see \cite[3.3, 3.4]{GR}). For $\Gamma_0(\nn)$, it is a deep result of Gekeler--Nonnengardt that the map $j$ is bijective: see \cite[Thm. 3.3]{GN}.

\medskip
We extend the isomorphism of Lemma~\ref{lemm-isomSMcochsurZ} to characteristic-$p$ coefficients.
\begin{lemm}\label{lemm-isomcarp} Let $\nn\lhd A$.
	\begin{enumerate}
		\item There are canonical isomorphisms of $\Fp$-vector spaces
		$$\SM_0(\nn,\Z)\otimes_{\Z} \Fp \xrightarrow{\sim} \Harm_{0,0}(\nn,\Fp) \xrightarrow{\sim}  M_{2,1}^{0,0}(\nn,\Fp)$$
		which are compatible with Hecke operators.
		\item There are canonical isomorphisms of $\Cinf$-vector spaces
		$$\SM_0(\nn,\Z) \otimes_{\Z}\Cinf  \xrightarrow{\sim}  \Harm_{0,0}(\nn,\Cinf) \xrightarrow{\sim}  M_{2,1}^{0,0}(\nn)$$
		which are compatible with Hecke operators.
	\end{enumerate}
\end{lemm}
\begin{proof}
	\begin{enumerate}
		\item By Lemma~\ref{lemm-isomSMcochsurZ}, we have an exact sequence of $\Z$-modules
		$$ 0 \to \SM_0(\nn,\Z)\tors \to \SM_0(\nn,\Z) \to \Harm_{0}(\nn,\Z) \to 0.$$
		Tensoring with $\Fp$, we get the exact sequence
		$$ \SM_0(\nn,\Z)\tors \otimes_{\Z}\Fp \to \SM_0(\nn,\Z)\otimes_{\Z}\Fp \to \Harm_{0}(\nn,\Z)\otimes_{\Z}\Fp \to 0.$$
		However the torsion of $\SM_0(\nn,\Z)$ is annihilated by $q^2-1$ (Proposition~\ref{prop-smtorsion}) thus $\SM_0(\nn,\Z)\tors \otimes_{\Z}\Fp = \{ 0 \}$. Also recall the canonical isomorphism $\Harm_{0}(\nn,\Z)\otimes_{\Z}\Fp \simeq \Harm_{0,0}(\nn,\Fp)$ from the definition of $\Harm_{0,0}$. The conclusion then follows from Theorem~\ref{theo-GRisom} and the compatibility of all isomorphisms with Hecke operators.
		\item We extend the scalars from $\Fp$ to $\Cinf$ starting from the previous isomorphisms and use the fact that $M_{2,1}^{0,0}(\nn,\Fp)$ is an $\Fp$-structure on $M_{2,1}^{0,0}(\nn)$.
		\end{enumerate}
\end{proof}

\begin{rema}
	In view of classical modular symbols, it is more natural to think of modular symbols as the \emph{dual space} of modular forms with respect to the integration pairing, which is Hecke-compatible. For modular symbols over $K$, there is a similar integration pairing on the Bruhat-Tits tree $\BT$. It provides an exact sequence
	$$ \SM_0(\nn,\Z)/\SM_0(\nn,\Z)\tors \to  \Hom(\Harm_{0}(\nn,\Z),\Z) \to \Phi_\infty \to 0$$
	where $\Phi_\infty$ is the finite group of connected components of the Néron model of the Jacobian variety $J(X_0(\nn))$ at the place $\infty$ (\cite[Thm.14]{TeiMS}). However it is not clear whether Lemma~~\ref{lemm-isomcarp} can be derived from this exact sequence. There are examples where the group $\Phi_\infty \otimes_{\Z}\Fp$ is non-zero (e.g. when $q=2$ and $\nn$ is the prime ideal generated by $T^4+T^3+1$, the group $\Phi_\infty$ has order $160$, see \cite[5.3.3]{Gekcuspidaldivisorclass1997}).
\end{rema}

\subsection{Manin-Teitelbaum presentation}
We recall Teitelbaum's finite presentation for the space of modular symbols \cite{TeiMS} which is the counterpart of Manin's presentation for classical modular symbols. The projective line $\PP^{1}(A/\nn)$ over the finite ring $A/\nn$ consists of equivalence classes $(u:v)$ of pairs $(u,v) \in A \times A$ with $(u)+(v)+\nn= A$. Recall that two pairs $(u_1,v_1)$ and $(u_2,v_2)$ are equivalent if there exists $w\in A$ such that $(w)+\nn =A$, $u_2 \equiv w u_1 \bmod \nn$ and $v_2 \equiv w v_1 \bmod \nn$. 

Note that any class in $\PP^{1}(A/\nn)$ may be written $(u:v)$ with $\gcd(u,v)=1$. The left-action of $\Gamma_0(\nn)$ on $\GL_2(A)$ provides the bijection
$$ \begin{array}{rcl}
	\Gamma_0(\nn) \bs \GL_2(A) & \longrightarrow & \PP^{1}(A/\nn) \\
	\Gamma_0(\nn) \begin{pmatrix} a &b\\u&v\end{pmatrix} &\longmapsto &(u:v).\end{array}$$	
Let $g = \smallpmatrix{a}{b}{u}{v} \in \GL_2(A)$ be a matrix corresponding to $(u:v)$. The map
$$ \begin{array}{rcl}
	\PP^{1}(A/\nn) & \longrightarrow & \SM(\nn,\Z) \\
	(u:v) & \longmapsto & [g(0),g(\infty)] = \left[ \frac{b}{v},\frac{a}{u} \right]
\end{array}$$
extends to a $\Z$-linear map
$$\xi : \Z[\PP^{1}(A/\nn)] \longrightarrow \SM(\nn,\Z).$$
Continued fraction expansions show that $\xi$ is surjective thus the family of modular symbols $(\xi(u:v))_{(u:v) \in \PP^{1}(A/\nn)}$ generates the space $\SM(\nn,\Z)$ over $\Z$. Relations among this generators are described by the next statement.

\begin{theo}[{\cite[Thm. 21]{TeiMS} and \cite[Thm.~4.6]{ArmBaseSM}}]\label{theo-presentationSM} The $R$-module $\SM(\nn,R)$ is isomorphic via $\xi$ to the quotient of the free $R$-module $R[\PP^{1}(A/\nn)]$  by the submodule generated, for all $(u:v)\in\PP^{1}(A/\nn)$, by the elements
	\begin{align*}
		&(u:v) + (-v:u) & \text{($2$-term relations)}, \\	
		&(u:v) + (v:-u-v) + (-u-v:u) & \text{($3$-term relations)} ,\\
		\forall (\delta_1,\delta_2) \in \Fq\et\times\Fq\et,\quad  &(u:v) - (\delta_1 u: \delta_2 v)&\text{(diagonal relations).} 
	\end{align*}	
\end{theo}	

\smallskip
When $\nn$ is prime, there exists a large explicit family of linearly independent generators which we recall now. It has no known counterpart for classical modular symbols and will be our essential tool to prove the non-perfectness of pairings in Section~\ref{section-nonzeroelts}.

\begin{theo}[{\cite[Thm.~5.16 (ii)]{ArmBaseSM}}]\label{theo-famindepSM}
	Let $\nn$ be a prime ideal of $A$. The family
	$$ \{ \xi(1:0) \} \cup \{ \xi(u:v) \mid (u,v) \in A_+ \times A_+, \, \deg v < \deg u < \deg(\nn)/2,\, \gcd(u,v)=1 \}$$	
	is linearly independent in $\SM(\nn,R)$. Moreover in this family, the modular symbol $\xi(1:0)$ is not cuspidal whereas all the others are cuspidal.	
\end{theo}	
Although we will not need it, let us mention that when $\deg \nn$ is odd this family turns out to be an $R$-basis of $\SM(\nn,R)$ (see \cite[Thm.~5.16 (ii)]{ArmBaseSM}).

\smallskip
Finally we recall a formula for the action of Hecke operators in terms of generators of the finite presentation. This is a function field analog of a result of Merel \cite{MerelUniversalFourier}.

\begin{prop}[{\cite[Thm~6.1]{ArmBaseSM}}]\label{prop-formuleMerel}
	Let $\pp$ be an ideal of $A$. Let 
	$$ \Sigma_\pp = \left\{ \begin{pmatrix}a&b\\c&d\end{pmatrix} \in M_2(A) \grandmid (a,d) \in A_+\times A_+,\, \deg a > \deg b,\, \deg d > \deg c,\, (ad-bc)=\pp \right\}.$$
	For any $(u:v)\in\PP^{1}(A/\nn)$, we have the equality in $\SM(\nn,R)$
	$$ T_\pp \, \xi(u:v) = \sum_{\smallpmatrix{a}{b}{c}{d} \in \Sigma_\pp} \xi (au+cv:bu+dv)$$	
	where the sum is restricted to the terms which are well-defined in $\PP^{1}(A/\nn)$, i.e. $(au+cv)+ (bu+dv)+\nn = A$.
\end{prop}
The set of matrices $\Sigma_\pp$ is finite. For instance if $\deg \pp=1$ and denoting by $P$ the monic generator of $\pp$, we have (see \cite[6.2]{ArmBaseSM})
\begin{equation}\label{eq-matheckedeg1}
	\Sigma_\pp = \left\{ \begin{pmatrix} P&\lambda \\ 0 & 1 \end{pmatrix}, \begin{pmatrix} 1&0 \\ \lambda & P \end{pmatrix} \grandmid \lambda \in \Fq \right\}.
\end{equation}

\begin{rema}\label{rema-TPdeg2}
	When $\deg \pp=2$, the complete list of matrices in $\Sigma_\pp$ may be found in \cite[Lemma~6.5]{ArmBaseSM}.
\end{rema}

\section{Non-zero kernel element}\label{section-nonzeroelts}

Our aim is to prove Theorem~\ref{th-main} by a computation with modular symbols. To shorten notation, we will write $\SM$ (resp. $\SM_0$) for the $\Z$-module $\SM(\nn,\Z)$ (resp. $\SM_0(\nn,\Z)$). For any modular symbol $\xi(u:v) \in \SM$, let $x(u:v)$ be its image in $\SM/p\SM \simeq \SM\otimes_{\Z} \Fp$.

\begin{prop}\label{prop-main}
Suppose that the ideal $\nn$ has no degree-$1$ factor. Let $u \in A$ be any monic polynomial of degree $1$. For any $\pp\lhd A$, let $P$ denote its monic generator in $A$. The modular symbol $x(u:1)$ is cuspidal and we have in $\SM/p\SM$:
	\begin{equation}\label{prop-calculSM-F1} \sum_{\deg \pp \leq 1 } T_\pp\, x(u:1) = \sum_{\deg \pp = 1} \sum_{\substack{v\in A_{+\leq 1} \\[.05cm] v \neq u,\, v \neq P}} x(Pu:v).
	\end{equation}
\end{prop}
\begin{proof}
First let us show that $\xi(u:1) \in \SM_0$. Take $g = \left(\begin{smallmatrix}1&0\\u&1\end{smallmatrix}\right) \in \GL_2(A)$. We have $\xi(u:1) = [g(0),g(\infty)] = [0,\frac{1}{u}]$. Let $N$ denote the monic generator of $\nn$. Since $\gcd(u,N)=1$ by assumption on $\nn$, there exists a matrix $\gamma = \left(\begin{smallmatrix}\alpha &1\\\beta N&u\end{smallmatrix}\right)\in\Gamma_0(\nn)$. We have $\gamma(0) = \frac{1}{u}$ so the cusps $0$ and $\frac{1}{u}$ are $\Gamma_0(\nn)$-equivalent, which proves $[0,\frac{1}{u}] \in \SM_0$.
		
\smallskip	
Now let us prove~\eqref{prop-calculSM-F1}. Using Proposition~\ref{prop-formuleMerel} when $\deg \pp=1$ (see \eqref{eq-matheckedeg1}), we have
		\begin{equation}\label{eq-propcalculSM1} T_\pp \xi(u:1) = \sum_{\lambda \in \Fq}  \xi (Pu:\lambda u +1) + \sum_{\lambda \in \Fq} \xi (u+\lambda : P).
		\end{equation}
		\begin{enumerate}
			\item Suppose that $P\neq u$. Since $P-u$ is a non-zero element in $\Fq$, the first sum in \eqref{eq-propcalculSM1} is 
			$$\sum_{\lambda \in \Fq} \xi(Pu:\lambda u+1) = \xi(Pu:1) + \xi(Pu:\frac{P}{P-u}) + \sum_{\substack{\lambda \neq 0 \\ \lambda (P-u) \neq 1}} \xi(Pu:\lambda u+1).$$
			Moreover by assumption on $\nn$, we have $\pp+\nn=A$ so $(Pu : \frac{P}{P-u}) = (u:\frac{1}{P-u})$ in $\PP^1(A/\nn)$. Using diagonal relations, we get for all $\lambda \in \Fq\et$,
			$$ \xi(Pu : \frac{P}{P-u}) = \xi(u:\frac{1}{P-u}) = \xi(u:1) \quad\text{and}\quad \xi(Pu:\lambda u+1) = \xi(Pu:u+\lambda^{-1}).$$
			The second sum in \eqref{eq-propcalculSM1} involves $\xi(u+\lambda:P)$ which we rewrite as
			\begin{align*}
				\xi(u+\lambda : P) &= \xi(u+\lambda:-P) & \text{(diagonal rel.)} 
				\\& =-\xi(-P:-u-\lambda+P) - \xi(-u-\lambda+P:u+\lambda) & \text{($3$-term rel.)} \\
				& = -\xi(P:-u-\lambda+P)+\xi(-u-\lambda:-u-\lambda+P) & \text{(diagonal and $2$-term rel.)} \\
				& =-\xi(P:-u-\lambda+P) + \xi(u+\lambda:-u-\lambda+P) & \text{(diagonal rel.)}
			\end{align*}
		where $-u-\lambda+P\in\Fq$ since both $P$ and $u$ are monic of degree~$1$.
		
		If $-u-\lambda+P\neq 0$, using diagonal relations we get
		$$ \xi(u+\lambda :P)=-\xi(P:1)+\xi(u+\lambda:1).$$
				
		If $-u-\lambda+P=0$, since each of the degree-$1$ polynomials $P$ and $u+\lambda$ are coprime to $\nn$, we obtain
		$$ \xi(u+\lambda:P) = -\xi(P:0)+\xi(u+\lambda:0)=-\xi(1:0)+\xi(1:0)=0.$$
		Finally when $P\neq u$, we have
		\begin{align*} T_\pp \xi(u:1) &= \xi(Pu:1) + \xi(u:1) + \sum_{\substack{\lambda \neq 0 \\ \lambda (P-u)\neq 1}} \xi(Pu:u+\lambda^{-1}) + \sum_{\lambda \neq P-u} ( - \xi(P:1) + \xi(u+\lambda:1)) \\
			& = \xi(Pu:1) + \xi(u:1) -(q-1) \xi(P:1) +\sum_{\substack{\lambda \neq 0 \\ \lambda (P-u)\neq 1}}   \xi(Pu:u+\lambda^{-1}) + \sum_{\lambda \neq P-u} \xi(u+\lambda:1).
			\end{align*}
		Modulo $p$, we obtain
\begin{multline}\label{eq-calculSM1}
	T_\pp x(u:1) = x(Pu:1) + x(u:1) + \sum_{\substack{\lambda \neq 0 \\ \lambda (P-u)\neq 1}}  x(Pu:u+\lambda^{-1}) + \sum_{\lambda\in\Fq} x(u+\lambda:1).
\end{multline}
	\item Suppose that $P=u$. Similarly we have from \eqref{eq-propcalculSM1}
	$$T_\pp \xi(u:1) = T_u \xi(u:1) = \sum_{\lambda \in \Fq}  \xi(u^2 : \lambda u+1) + \sum_{\lambda \in \Fq} \xi(u+\lambda : u).$$
	When $\lambda\neq 0$, a diagonal relation gives $\xi(u^2 : \lambda u+1) = \xi(u^2 : u+\lambda^{-1})$. Moreover for all $\lambda \in \Fq$, we have
	\begin{align*}
		\xi(u+\lambda:u) & = \xi(u+\lambda :-u) & \text{(diagonal rel.)} \\
		& = -\xi(-u:-\lambda) - \xi(-\lambda:u+\lambda) & \text{($3$-term rel.)}\\
		& = -\xi(u:\lambda) + \xi(-u-\lambda:-\lambda) & \text{(diagonal and $2$-term rel.)} \\
		&= -\xi(u:\lambda) + \xi(u+\lambda:\lambda) & \text{(diagonal rel).}
\end{align*}	
For $\lambda\neq 0$, by diagonal relations we obtain $\xi(u+\lambda : u) = -\xi(u:1) + \xi(u+\lambda:1)$ and for $\lambda=0$, we have $\xi(u+\lambda:u) = -\xi(u:0) + \xi(u:0)=0$.

Finally when $P=u$, we obtain
\begin{align*} T_\pp \xi(u:1)&= \xi(u^2:1)+\sum_{\lambda\in\Fq\et} \xi(u^2 : u+\lambda^{-1}) + \sum_{\lambda \in \Fq\et} (-\xi(u:1)+\xi(u+\lambda:1)) \\
	& = \xi(u^2:1) -(q-1) \xi(u:1) + \sum_{\lambda\in\Fq\et} \xi(u^2 : u+\lambda) + \sum_{\lambda\in\Fq\et} \xi(u+\lambda:1).
\end{align*}
	Modulo $p$, we get
	\begin{equation}\label{eq-calculSM2}
		T_\pp\, x(u:1) = x(u^2:1) + \sum_{\lambda\in\Fq\et}  x(u^2:u+\lambda) +\sum_{\lambda\in\Fq} x(u+\lambda:1).
	\end{equation}
		\end{enumerate}
	Finally we combine \eqref{eq-calculSM1} and \eqref{eq-calculSM2} to compute $$\sum_{\deg \pp \leq 1} T_\pp x(u:1) = x(u:1) + \sum_{\deg \pp=1} T_\pp x(u:1).$$ The term $\sum_{\lambda \in \Fq} x(u+\lambda:1)$ appears $q$ times, hence vanishes. In the remaining expression, the term 
	$x(u:1)$ also appears $q$ times, hence vanishes.
Finally we get 
	$$ \sum_{\deg \pp \leq 1} T_\pp x(u:1) = \sum_{\deg \pp=1} x(Pu:1) + \sum_{\deg \pp=1} \sum_{\substack{\lambda \neq 0 \\ \lambda (P-u)\neq 1}} x(Pu:u+\lambda^{-1}).$$
	Since $\{ u+\lambda^{-1} \mid \lambda\neq 0, \lambda(P-u)\neq 1 \}$ coincides with the set of monic polynomials of degree~$1$ which are distinct from $u$ and $P$, we get the result.
\end{proof}

The following statement proves Conjecture~\ref{conj-69} (Conjecture~6.9 in \cite{ArmCoeffDMF}).
\begin{theo}\label{th-main}
Suppose that the ideal $\nn$ is prime of degree~$\geq 5$.
The element
$$ \sum_{\deg \pp \leq 1} T_\pp $$
is non-zero in $\TT(\nn,\Fp)$. 
\end{theo}
\begin{proof}
By Lemma~\ref{lemm-isomcarp}, it is equivalent to show that $ \sum_{\deg \pp \leq 1} T_\pp $ is non-zero in $\End_{\Fp}( \SM_{0}/p\SM_{0})$. Let $$ S = \{ (u:v) \in \PP^{1}(A/\nn) \mid (u,v)\in A_+\times A_+,\,\deg v < \deg u \leq 2,\, \gcd(u,v)=1 \}.$$
Assume that $\nn$ is prime of degree~$\geq 5$. By Theorem~\ref{theo-famindepSM}, the family $(x(u:v))_{(u:v)\in S}$ is linearly independent in $\SM(\Fp)(=\SM/p\SM)$ therefore in $\SM_0/p\SM_0$. 
In Proposition~\ref{prop-main}, observe that the right-hand side of \eqref{prop-calculSM-F1} is a linear combination of elements of this family, with at least one non-zero coefficient. Therefore $ \sum_{\deg \pp \leq 1} T_\pp $ is non-zero in $\End_{\Fp}( \SM_{0}/p\SM_{0})$.
\end{proof}

\begin{rema}
	When $\nn$ is prime of degree~$3$, the element $\sum_{\deg \pp \leq 1} T_\pp$ is zero in $\TT(\nn,\Fp)$ because of Proposition~\ref{prop-pairingDMF-casparfait}. The case $\deg \nn=4$ will be discussed in the next section.	
\end{rema}

\begin{coro}\label{coro-main}
Suppose that the ideal $\nn$ is prime of degree~$\geq 5$. The pairings $\pairing_{\DMF}$, $\pairing_{\HC,\Fp}$ and $\pairing_{\HC,\Z}$ are not perfect over $\Cinf$, $\Fp$ and $\Z$ respectively.
\end{coro}	
\begin{proof}
Corollary~\ref{coro-coeffDMF} and Proposition~\ref{prop-eltnoyaucochaine} ensure that $\sum_{\deg \pp \leq 1} T_\pp$ is in the left-kernel of $\pairing_{\DMF}$ and $\pairing_{\HC,\Fp}$ respectively. Under the assumption on $\nn$, Theorem~\ref{th-main} proves that this element is non-zero in $\End_{\Fp}(M_{2,1}^{0,0}(\nn,\Fp))$, therefore in $\End_{\Cinf}(M_{2,1}^{0,0}(\nn))$. This proves the assertion on $\pairing_{\DMF}$ and $\pairing_{\HC,\Fp}$.
\par If the pairing $\pairing_{\HC,\Z}$ were perfect over $\Z$, it would give $\TT(\nn,\Z) \simeq \Hom(\Harm_{0}(\nn,\Z),\Z)$ and, by reduction modulo~$p$, an isomorphism $\TT(\nn,\Fp) \simeq \Hom(\Harm_{0,0}(\nn,\Fp),\Fp)$ induced by $\pairing_{\HC,\Fp}$. This would contradict Theorem~\ref{th-main}.
\end{proof}

\section{Experimental data and open questions}\label{sec-expdata}

To explore the situation beyond Theorem~\ref{th-main} and in view of Corollary~\ref{coro-coeffDMF} and Proposition~\ref{prop-eltnoyaucochaine}, we have computed the elements $$\somme_{\leq 1} = \sum_{\deg \pp \leq 1} T_\pp \text{ and } \somme_d = \sum_{\deg \pp = d} T_\pp$$ in the Hecke algebra $\TT(\nn,\Fp)$. Computations were done using SageMath \cite{sagemath}. For simplicity, our algorithms have been written when $q$ is a prime number and $\nn$ a prime ideal, hence the assumptions in the following claim.
\begin{claim}\label{claim-numerical}Let $q$ be a prime number and $\nn$ a prime ideal.
\begin{enumerate}
	\item\label{item-claim1} If $\deg \nn=4$, we have $\somme_{\leq 1} = 0 $ and $\somme_{2}=0$.
	\item\label{item-claim2} If $\deg \nn\geq 5$, we have $\somme_2\neq 0$. Moreover:
	\begin{enumerate}
		\item If $\deg \nn \in \{5,6\}$, we have $\somme_2 = -\somme_{\leq 1}$.
		\item If $\deg \nn \geq 7$, the elements $\somme_{\leq 1}$ and $\somme_{2}$ are $\Fp$-linearly independent.
	\end{enumerate}
\item\label{item-claim3} If $\deg \nn \geq 5$ and $d\geq 3$, we have $\somme_d=0$.
\end{enumerate}
\end{claim}
This claim is supported by computations on many examples for $q\in \{2,3,5\}$ and $\deg \nn \in \llbracket 4,8  \rrbracket$ with $d\in \llbracket 1,5\rrbracket$. 

The first claim when $\deg \nn=4$ was also stated in \cite[Question~6.8]{ArmCoeffDMF} and it suggests that Theorem~\ref{th-main} cannot be extended to this case. Moreover recall that $\somme_{d}=0$ for any $d\geq 3$ by an argument on the support of cuspidal harmonic cochains, see \cite[Lemma~6.7]{ArmCoeffDMF}. Combined with the first claim, it suggests that none of the kernel elements $(\somme_d)_{d\geq 2}$ from Proposition~\ref{prop-eltnoyaucochaine} are non-zero when $\nn$ is prime of degree~$4$. Thus we have found no experimental obstruction to the perfectness of the pairings $\pairing_{\DMF}$ and $\pairing_{\HC,\Fp}$ in this case.

When $\deg\nn \geq 7$, Claim~\ref{claim-numerical} suggests that the dimension of the left-kernel of $\pairing_{\HC,\Fp}$ could be at least $2$.

\smallskip
In view of the pairing $\pairing_{\DMF}$, we also have computed examples of kernel elements with polynomial coefficients. We have focused on
$$ \somme_{2,A}= \sum_{P \in A_{+2}} P T_P \quad \in \TT(\nn,\Cinf)$$
which belongs to the left-kernel of $\pairing_{\DMF}$ when $q\geq 3$ by Corollary~\ref{coro-coeffDMF}. Our computations suggest that for any prime $\nn$ of degree $\geq 5$, this element is non-zero and that $(\somme_{2,A},\somme_{\leq 1},\somme_2)$ are linearly independent over $\Fp$. In this setting, the left-kernel of $\pairing_{\DMF}$ could be of dimension at least $3$. 

\smallskip
Many questions about these pairings remain open. We mention a couple of them. 
\begin{itemize}
	\item When $\nn$ is prime of degree~$4$, are the pairings $\pairing_{\HC,\Fp}$, $\pairing_{\HC,\Z}$ and $\pairing_{\DMF}$ perfect?
	\item When $\nn$ is prime of degree $\geq 5$:
	\begin{itemize}\item Can we have a more precise description of the left-kernels of $\pairing_{\DMF}$ and $\pairing_{\HC,\Fp}$: generating families, dimensions? 
	\item  When $\nn$ is prime of degree $\geq 5$, the determinant of $\pairing_{\HC,\Z}$, with respect to any $\Z$-bases of $\TT(\nn,\Z)$ and $\Harm_{0}(\nn,\Z)$, is a positive power of~$p$, see Theorem~\ref{th-main-intro}. This power measures how far is the pairing from being perfect. It would be interesting to compute this exponent on examples and to understand if it is connected to other arithmetical or geometrical quantities attached to harmonic cochains or Drinfeld modular curves.
	\end{itemize}
	\item What can be proved about these pairings when the level $\nn$ is not a prime?
	\end{itemize}

\bibliographystyle{alpha}
\bibliography{biblio}
   
\end{document}